\documentclass[11pt]{amsart}
\usepackage{graphicx}
\usepackage{amsfonts}
\usepackage{amscd}
\usepackage{amssymb}
\usepackage{alltt}
\usepackage{amsmath}

\newtheorem{definition}[equation]{Definition}
\newtheorem{thm}[equation]{Theorem}
\newtheorem{lemma}[equation]{Lemma}
\newtheorem{conj}[equation]{Conjecture}

\parskip=\baselineskip

\def\op#1{{\operatorname{#1}}}
\newcommand{\ring}[1]{\mathbb{#1}}
\def\deltalat{\mathbb\delta}  
\def\delt{\delta_{\min}}
\def\rZ{{\ring{Z}}}
\def\rR{{\ring{R}}}
\def\rP{{\ring{P}}}

\def\aa{{\op{area}}}
\def\ta{{\tau}}
\def\mid{\,:\,}

\begin{document}

\title{On the Reinhardt Conjecture}
\author{Thomas C. Hales}
\email{hales@pitt.edu} \thanks{Research supported by NSF grants
  0503447 and 0804189.  This research was conducted in 2007 at the
  Hanoi Math Institute.  I thank the members there for their
  hospitality.  These results were presented at a
  conference in honor of L. Fejes T\'oth in Budapest, June 2008.  I
  thank Nate Mays for turning my attention to this problem}.

\begin{abstract}  
  In 1934, Reinhardt asked for the centrally symmetric convex domain
  in the plane whose best lattice packing has the lowest density.  He
  conjectured that the unique solution up to an affine transformation
  is the smoothed octagon (an octagon rounded at corners by arcs of
  hyperbolas).  This article offers a detailed strategy of proof.  In
  particular, we show that the problem is an instance of the classical
  problem of Bolza in the calculus of variations.  A minimizing
  solution is known to exist.  The boundary of every minimizer is a
  differentiable curve with Lipschitz continuous
  derivative.  If a minimizer is piecewise analytic, then it is a
  smoothed polygon (a polygon rounded at corners by arcs of
  hyperbolas).  To complete the proof of the Reinhardt conjecture, the
  assumption of piecewise analyticity must be removed, and the
  conclusion of smoothed polygon must be strengthened to smoothed
  octagon.
\end{abstract}

\maketitle


\centerline{to H\`a Huy Kho\'ai}

\section{Introduction}

{\narrower\it A contract requires a miser to make payment with a tray
  of identical gold coins filling the tray as densely as possible.
  The contract stipulates the coins to be convex and centrally
  symmetric.  What shape coin should the miser choose in order to part
  with as little gold as possible?}

Let $K$ be a centrally symmetric convex domain in the Euclidean plane.
If $\Lambda$ is a lattice such that the translates of $K$ under
$\Lambda$ have disjoint interiors, then the packing density of
$\Lambda+K$ is the ratio of the area of $K$ to the co-area of the
lattice $\Lambda$.  Let $\deltalat(K)$ be the maximum density of any
lattice packing of $K$.  A lattice realizing this density exists for
each $K$.

Let $\delt$ be the infimum of $\deltalat(K)$, as $K$ ranges over all
convex domains in the Euclidean plane.  Reinhardt proves that there
exists $K$ for which $\deltalat(K)=\delt$.  Reinhardt's problem is to
determine the constant $\delt$ and to describe $K$ explicitly for
which $\deltalat(K)=\delt$.

Reinhardt conjectured that $\deltalat(K) =\delt$ when $K$ is a
smoothed octagon.  The smoothed octagon is constructed by taking a
regular octagon and clipping the corners with hyperbolic arcs.  The
hyperbolic arcs are chosen so that the smoothed octagon has no
corners; that is, so that there is a unique tangent at each point of
the boundary.  The asymptotes of each hyperbola are lines extending
two sides of the regular octagon.  The density of the smoothed octagon
is
\[\deltalat(K) = \frac{8 - \sqrt{32} - \ln{2}}{\sqrt8 - 1}\approx 0.902414.\]

Reinhardt's original article contains many useful facts about his
conjecture.  The main facts from his article have been summarized in
Section~\ref{sec:rein}.  Beyond Reinhardt's article, various lower
bounds for $\delt$ have been published.  
See \cite{Juhasz:1983}, 
\cite{Mahler:1946}, \cite{Mahler:47a}, \cite{Mahler:47b}, \cite{Toth:1972:Lagerungen}, \cite{Tammela:1970},
\cite{Ennola-1961}.  The special case of centrally symmetric decagons is
considered in \cite{Ledermann:1949}.  Nazarov has proved the local optimality of
the smoothed octagon \cite{Nazarov}. The problem is discussed further in
\cite{Pach:1995}, where it is referred to as a ``famous conjecture.''  It is
also known that non-lattice packings of centrally symmetric convex
domains in the plane cannot have greater density than $\deltalat(K)$
\cite{Fejes-Toth:50}.  The Ulam conjecture, which is the corresponding
conjecture in three dimensions, posits the sphere as solution
\cite{Gardner:2001}.

In this article we take the following approach.  The boundary of any
optimal $K$ is a $C^1$ curve.  We express the boundary of $K$ in the
calculus of variations.  We will see that the problem is a special
instance of the classical problem of Bolza.  The circle is the unique
solution to the Euler-Lagrange equations (up to an affine
transformation), but an examination of second-order conditions shows
that the circle does not minimize.  This means that optimal $K$ is not
an interior point of the configuration space.

This leads to a study of the constraint that $K$ must be convex.
Piecewise analytic solutions have a well-defined (local) invariant,
called the rank.  Calculus of variations further reduces the problem
to the study of rank one.  We show that rank-one solutions $K$ are
structurally similar to the smoothed octagon.  In particular, $K$ is a
smoothed polygon, whose boundary consists of finitely many linear
segments connected by hyperbolic arcs.  We propose a nonlinear optimization
problem in a small number of variables over certain rank-one
configurations.  The successful solution of this nonlinear
optimization problem (and eliminating the assumption of piecewise
analyticity) would complete the proof of the Reinhardt conjecture.

\section{Reinhardt's Article}\label{sec:rein}

We give a brief review a series of lemmas in Reinhardt's original article.
The proofs are generally elementary.

\subsection{balanced hexagons}

\begin{definition}[balanced hexagon]
  We call a centrally symmetric hexagon $G$ a {\it balanced hexagon}
  of a centrally symmetric convex domain $K$ if $G$ contains $K$ and
  if the midpoint of each side of $G$ is a point on the boundary of
  $K$.  These six points on the boundary of $K$ are called the
  midpoints of $G$.
\end{definition}

\begin{lemma}\label{lemma:parallel} 
  Let $G$ be a balanced hexagon of a centrally symmetric convex domain
  $K$ without corners.  Then $G$ does not degenerate to a
  parallelogram.  That is, the six vertices are distinct \cite{Reinhardt:1934}.
\end{lemma}

\begin{lemma}\label{lemma:mid1} 
  Let $K$ be a centrally symmetric convex domain.  Each point of $K$
  is a midpoint of at most one balanced hexagon \cite[p.228]{Reinhardt:1934}.
\end{lemma}

\begin{lemma}\label{lemma:8G} 
  Let $G$ be a balanced hexagon of a centrally symmetric convex domain
  $K$.  Assume that the center of symmetry of $K$ is the origin.  Let
  $u_j$, $j\in \ring{Z}/6\ring{Z}$, be the midpoints of the sides of
  $G$ listed in cyclic order around the hexagon.  Then
 \begin{itemize}
 \item $u_j+u_{j+2}+u_{j+4}=0$.
 \item $u_{j+3} = -u_j$.
 \item The area of $G$ is $4/3$ the area of the hexagon 
     $H$ formed by the convex hull of $\{u_j\}$.
  \item The six segments from the origin to the six midpoints $u_j$ breaks
    $H$ into six congruent triangles.  In particular, the area of $G$
    is $8$ times the area of the triangle $\{0,u_j,u_{j+2}\}$.
  \end{itemize}
\end{lemma}

\begin{proof} If the vertices of the centrally symmetric hexagon $G$
  are $w_j$, with $w_{j+3} = - w_j$, then
\[
u_j = (w_j+w_{j+1})/2.
\]
The first two statements are then immediate.  The other statements appear in 
\cite[p.219,p.222]{Reinhardt:1934}.
\end{proof}

\subsection{miserly domains}

\begin{lemma}  There exists a centrally symmetric convex domain
$K$ for which $\deltalat(K) = \delt$.
\end{lemma}

\begin{proof} This follows by Blaschke's selection lemma \cite[p.220]{Reinhardt:1934}.
\end{proof}
\begin{definition}[miserly domain]
  Any centrally symmetric convex domain $K$ that realizes the lower
  bound $\deltalat(K) = \delt$ is called a {\it miserly
    domain}. 
\end{definition}

\begin{lemma}\label{lemma:221} 
  If $K$ is a miserly domain, then it has no corners.  That is, there
  is a unique tangent through each point on the boundary.
  \cite[p.221]{Reinhardt:1934}
\end{lemma}

\begin{lemma}\label{lemma:mid-min} 
\begin{itemize}
\item Let $K$ be a centrally symmetric domain without corners.  Assume
  that for each point $p$ on the boundary of $K$ there exists a
  balanced hexagon $G_p$ on which $p$ is a midpoint.  Assume further
  that the area of $G_p$ is independent of $p$.  Then $K$ has no
  other balanced hexagons and
\begin{equation}\label{eqn:density}
\deltalat(K) = \aa(K)/\aa(G)
\end{equation}
for every balanced hexagon $G$ of $K$.
\item If $K$ is any miserly domain, then it satisfies the assumptions
  of the first part of the lemma.
\end{itemize}
\end{lemma}
\begin{proof} The facts asserted without proof in this proof appear in
  \cite[pp.219--222]{Reinhardt:1934}.  Let $K$ be a centrally symmetric convex
  domain in the plane.  Let $G$ be a smallest centrally symmetric
  hexagon that contains $K$.  Such a hexagon exists, and
  $\deltalat(K)$ equals the ratio of the area of $K$ to that of
  $G$. 
  Call any such hexagon a {\it fitting} hexagon.  Every fitting
  hexagon of $K$ is a balanced hexagon. 
  By Lemma~\ref{lemma:mid1}, there are no balanced hexagons other than
  the $G_p$.  Hence $G=G_p$ for some $p$.  The first part of the lemma
  now follows.

  Now let $K$ be a miserly domain.  Reinhardt proves that each
  boundary point $p$ of $K$ lies on a balanced hexagon that is also a
  fitting hexagon of $K$, although $p$ is not necessarily a midpoint
  of the balanced hexagon.  
  Next, he shows that each boundary point of $K$ is in fact the
  midpoint of a balanced hexagon that is also a fitting hexagon of
  $K$. 
  Since there are no other balanced hexagons, there are no other
  fitting hexagons.  The set of balanced hexagons coincides with the
  set of fitting hexagons.  All fitting hexagons have the same area.
  Thus, the assumptions of the first part of the lemma are all
  satisfied for a miserly domain.
\end{proof}

\section{the Boundary Curve}

\subsection{hexameral domains}

If we combine the properties of miserly domains that were established
by Reinhardt, we are led to the following definition.

\begin{definition}[hexameral domain]
We say that $K$ is a hexameral domain if the following conditions hold.
\begin{itemize}
\item $K$ is a centrally symmetric domain whose center of symmetry is the origin.
\item $K$ has no corners.
\item Each point on
the boundary of $K$ is a midpoint of a balanced
hexagon $G$.  Moreover, these balanced hexagons all have the same area.
\end{itemize}
\end{definition}

By the preceding lemmas, if $K$ is a miserly domain, then (after
recentering at the origin) it is a hexameral domain.  The packing density
$\deltalat(K)$ of a hexameral domain is computed by
Formula~(\ref{eqn:density}) and Lemma~\ref{lemma:8G}.  The smoothed
octagon and the circle are examples of hexameral domains.  The class
of hexameral domains is much larger than the class of miserly domains.
We consider the optimization problem of determining the miserly domains
within the class of hexameral domains.  If $K$ is a hexameral domain,
then each point of the boundary is a midpoint of a {\it unique}
balanced hexagon.

Let $K$ be a hexameral domain.  Give a continuous parametrization
$t\mapsto\sigma_0(t)$ of the boundary curve.  We follow the convention
of parametrizing the boundary in a counterclockwise direction.  Since
$K$ has no corners, we may assume that $\sigma_0$ is $C^1$.  At each
time $t$, there is a uniquely determined balanced hexagon with
midpoint $\sigma_0(t)$. Let the other midpoints be listed in
(counterclockwise) order as $\sigma_j(t)$, $j\in\rZ/6\rZ$.

If $u$ and $v$ are ordered pairs of real numbers, write
$u\land v$ for the $2\times 2$ determinant with columns $u$ and $v$.

\begin{lemma}\label{lemma:sigma2:c1} 
Let $K$ be a hexameral domain with $C^1$ boundary
parametrization $\sigma_0$.  Then the curves $\sigma_2$ and $\sigma_4$
are also $C^1$ parametrizations of the boundary, oriented in the same
way as $\sigma_0$.
\end{lemma}

\begin{proof} Reinhardt shows that the boundary parametrizations
  $\sigma_2,\sigma_4$ are continuous if $\sigma_0$ is continuous, and
  that they are oriented in the same way as $\sigma_0$
  \cite[p.222]{Reinhardt:1934}.  Let us check that $\sigma_2$ is $C^1$, whenever
  $\sigma_0$ is.  Since $K$ has no corners, the unit tangent $n_2(t)$
  to $\sigma_2(t)$, with the orientation given by $\sigma_0$, is a
  continuous function of $t$.  It is enough to check that the speed of
  $\sigma_2$ is continuous in $t$.  By  Lemma~\ref{lemma:8G},
  ${\sigma_0(t)}\land{\sigma_2(t)}$ is a fixed fraction of the area of
  the balanced hexagon, and does not depend on $t$.

  We claim that ${\sigma_0(t)}\land{n_2(t)}\ne 0$.  Let $H(t)$ be the
  hexagon given by the convex hull of $\{\sigma_j(t)\}$.  If
  ${\sigma_0(t)}\land{n_2(t)}=0$, then the tangent line to $\sigma_2$
  at $t$ contains the edge of $H(t)$ through $\sigma_2(t)$ and
  $\sigma_1(t)$.  Then also, ${\sigma'_0(t)}\land{\sigma_2(t)}=0$ and
  the tangent line to $\sigma_0$ lies along another edge of $H(t)$.
  This forces a corner at $\sigma_1(t)$, which is contrary to Lemma~\ref{lemma:221}.

This nonvanishing result and the fact that ${\sigma_0(t)}\land{\sigma_2(t)}$ 
is independent of $t$ imply that there exists a function
$v_2:\ring{R}\to\ring{R}$ such that 
  \begin{equation}\label{eqn:sB}
  {\sigma'_0(t)}\land{\sigma_2(t)} + {\sigma_0(t)}\land{n_2(t)} v_2(t) = 0.
  \end{equation}
The function $v_2(t)$ is the speed, and from the form of this equation,
it is necessarily continuous in $t$.
\end{proof}

\subsection{multi-curve}

There is no harm in rescaling a hexameral domain so that its balanced
hexagon has area $\sqrt{12}$, which is the area of a regular hexagon
of inradius $1$.  For this normalization, Lemma~\ref{lemma:8G} gives
 \begin{equation}\label{eqn:AB}
 {\sigma_j(t)}\land{\sigma_{j+2}(t)} = \sqrt{3}/2.
 \end{equation}

This suggests the following definition.

\begin{definition}[multi-point, multi-curve]
A function $u:\rZ/6\rZ\to\ring{R}^2$ such that 
  \begin{equation}\label{eqn:uA}
  u_j+u_{j+2}+u_{j+4} = 0,\quad u_{j+3} = -u_j,\quad  {u_j}\land{u_{j+2}}=\sqrt{3}/2
  \end{equation}
is called a {\it multi-point}. 
An indexed set of $C^1$ curves
\[
\sigma:\ring{Z}/6\ring{Z} \times [t_0,t_1]\to \ring{R}^2
\]
is a {\it  multi-curve} if for all $t\in[t_0,t_1]$, $\sigma(t)$ is a multi-point.
That is,
\begin{itemize}
\item $\sigma_j(t) + \sigma_{j+2}(t) + \sigma_{j+4}(t) = 0$,
\item $\sigma_{j+3}(t) = -\sigma_j(t)$,
\item ${\sigma_j(t)}\land{\sigma_{j+2}(t)}=\sqrt{3}/2$.
\end{itemize}
\end{definition}

By differentiation, a multi-curve also satisfies for all $j$:
\begin{equation}\label{eqn:sigma'}
{\sigma_j(t)}\land{\sigma'_{j+2}(t)} 
+ {\sigma'_j(t)}\land{\sigma_{j+2}(t)} = 0.
\end{equation}
The boundary of a hexameral domain admits a parametrization as a
triple curve.  The converse does not hold because a multi-curve has no
convexity constraint and no constraint for the curves $\sigma_j$ to
fit seamlessly into a simple closed curve containing the origin in the
interior.

\subsection{Lipschitz continuity}

\begin{lemma}\label{lemma:s0:lip}
  Let $K$ be a miserly domain and let $\sigma_j$ be a multi-curve
  parametrization on the boundary of $K$.  Assume that $\sigma_0$ is
  parametrized by arclength $s$.  Then $\sigma_0'$ is Lipschitz
  continuous.
\end{lemma}

\begin{proof} For each $s$, let $H_s$ be the hyperbola through
  $\sigma_0(s)$ whose asymptotes are the lines in direction
  $\sigma'_j(s)$ through $\sigma_j(s)$, for $j=\pm 1$.  By
  Reinhardt~\cite[p.220]{Reinhardt:1934}, near $\sigma_0(s)$, the an arc of $H_s$
  lies inside $K$.  As $s$ varies, by continuity over the compact
  boundary, the curvatures of the hyperbolas $H_s$ at $\sigma_0(s)$
  are bounded above by some $\kappa\in \ring{R}$.  This means that a
  disk of fixed curvature $\kappa$ can be placed locally in $K$ at each point
  $\sigma_0(s)$ so that $\sigma_0'(s)$ is tangent to the disk.  The
  curve $\sigma_0$ is constrained on the other side by convexity, so
  that $\sigma_0$ is wedged between the tangent line to $\sigma_0$ at $s$
  and the disk.

  If we parametrize the curve by arclength, then $\sigma_0'(s)$ has unit length.
  Lipschitz continuity now follows from this bound $\kappa$ on the curvature.
\end{proof}


\begin{lemma}\label{lemma:sj:lip}
  Let $K$ be a miserly domain and let $\sigma_j$ be a multi-curve
  parametrization on the boundary of $K$.  Assume that $\sigma_0$ is
  parametrized by arclength $s$.  Then $\sigma_j'$ is Lipschitz
  continuous for all $j$.
\end{lemma}

\begin{proof}
  By evident symmetries, it is enough to consider $j=2$.  Let $t$ be
  the arclength parameter for the curve $\sigma_0$ and let $s$ be the
  arclength parameter for the curve $\sigma_2$. We consider $s$ as a
  function of $t$.  By Lemma~\ref{lemma:sigma2:c1}, the function $s$
  is $C^1$.

We show that $ds/dt$ is Lipschitz continuous function of $t$.  In
fact, $ds/dt=v_2(t)$, given by (\ref{eqn:sB}).  The coefficient
$\sigma_0(t)\land n_2(t)$ of $v_2(t)$ in (\ref{eqn:sB}) is nonzero and
by continuity is bounded away from $0$.  Thus, the Lipschitz
continuity of $v_2$ follows from the Lipschitz continuity of the other
functions $\sigma_0'$, $\sigma_2$, $\sigma_0$, and $n_2$ in that
equation.

Now we show that $\sigma_2'$ is a Lipschitz continuous function of $t$.
Write
\[
\sigma_2'(t) = \frac{d\sigma_2}{d s} \frac{ds}{dt}.
\]
The first term on the right is Lipschitz continuous by Lemma~\ref{lemma:s0:lip}.  The
second term on the right has just been shown to be Lipschitz continuous.  Hence the result.
\end{proof}

Since $\sigma_j'$ is Lipschitz continuous, Rademacher's theorem
implies that $\sigma_j'$ is differentiable almost everywhere (or
directly we have that the angular argument of $\sigma_j'$ is
monotonic, hence differentiable almost everywhere).  Thus, we may
express the convexity constraint locally at $\sigma_j(t)$ by a second
derivative:
\begin{equation}\label{eqn:curvature''}
\sigma_j'(t)\land \sigma_j''(t) \ge 0.
\end{equation}

\subsection{special linear group}

The special linear group
$SL_2(\ring{R})$ acts on $\ring{R}^2$ by linear transformations and preserves the
wedge product:
\[
    {g u}\land{ g v}= \det(g)({u\land v}) .
\]
Conversely any affine transformation fixing the origin and fixing some
${u}\land{v}\ne 0$ must be given by some $g\in SL_2(\ring{R})$.

The group $SL_2(\ring{R})$ acts on the data of the Reinhardt problem,
on the set of miserly domains, on the set of multi-curves, and so
forth.

Given a multi-curve $\sigma$ and multi-point $u$, there exists a
unique $C^1$ curve $\phi:[t_0,t_1]\to SL_2(\ring{R})$, 
such that
   \begin{equation}\label{eqn:sigma-phi}
   \sigma_j(t) = \phi(t) u_j
   \end{equation}
for $j\in \ring{Z}/6\ring{Z}$.

The transformed multi-curve $\phi(t_0)^{-1} \sigma_j$
starts at 
$\sigma_j(t_0) = u_j$.
It is often convenient to use the multi-point formed by roots of unity:
\begin{equation}\label{eqn:roots}
u^*_j = \exp(\pi i j/3),\quad { i = \sqrt{-1} }.
\end{equation}
In particular, any hexameral domain is equivalent under
$SL_2(\ring{R})$ to a hexameral domain that starts at the multi-point
$u^*$ on the unit circle.  We call this a {\it circle representation}
of the hexameral domain or multi-curve.

Let $\sigma$ be a multi-curve.  Define $X:[t_0,t_1]\to\mathfrak{gl}_2(\rR)$
by
\[
\sigma'_j(t) = X(t)\sigma_j(t),\hbox{ for } j = 0,2.
\]
Then also,
\[
\sigma'_j(t) = X(t)\sigma_j(t),\hbox{ for } j\in\rZ.
\]
and
\begin{equation}\label{eqn:Xt}
\phi'(t) = X(t) \phi(t),
\end{equation}
where $\phi$ is given by Equation~(\ref{eqn:sigma-phi}).  Equation
(\ref{eqn:sigma'}) implies that $X(t)\in\mathfrak{sl}_2(\rR)$, the Lie
algebra of $SL_2(\rR)$.  The tangent lines to the curves $\sigma_j$
are determined by the image of $X(t)$ in the projective space
$\rP(\mathfrak{sl}_2(\rR))$ over the vector space
$\mathfrak{sl}_2(\rR)$.

If we transform $\sigma_j$ to $g\sigma_j$, for some $g\in SL_2(\rR)$,
then $X(t)$ transforms to $\op{Ad}\, g\, X = g X(t) g^{-1}\in \mathfrak{sl}_2(\rR)$.
By Lemma~\ref{lemma:sj:lip}, if $\sigma_0$ is parametrized by arclength, then
$X$ is Lipschitz continuous.

We have seen that the parametrized boundary of a hexameral domain
determines a curve in $SL_2(\ring{R})$.  Conversely, a curve in $SL_2$
determines a hexameral domain in the following sense.

\begin{lemma}\label{lemma:sl2-rein}  
  Let $K$ be a centrally symmetric convex domain (with center of
  symmetry $0$) with a multi-point $u$ on the boundary.  Let
  $\phi:[t_0,t_1]\mapsto SL_2(\ring{R})$ be a $C^1$ curve.  Define
  curves $\sigma_j$ by Equation~(\ref{eqn:sigma-phi}).  Assume that
  $\sigma_j$ parametrizes the boundary of $K$, for
  $j\in\ring{Z}/6\ring{Z}$.
Then
$K$ is a hexameral domain.
\end{lemma}

\begin{proof} 
  We check the balanced hexagon condition.  At time $t$, let $w_j(t)$
  be the point of intersection of the tangent line to $\sigma_j(t)$
  with the tangent line to $\sigma_{j+1}(t)$.  The condition that
  $w_j(t)$ are the vertices of a balanced hexagon generates a system
  of six linear equations and three unknowns.  Consistency of this
  system of equations imposes three constraints:
   \[
   \begin{array}{lll}
   0 &= \sigma_0(t) + \sigma_2(t) + \sigma_4(t),\\
   0 &= {\sigma_0(t)}\land{\sigma_2'(t)}+ {\sigma_0'(t)}\land{\sigma_2(t)}.\\
   \end{array}
   \]
   Integrating the final constraint, gives that
   ${\sigma_0(t)}\land{\sigma_2(t)}$ is constant.  These conditions
   hold for a curve coming from $\phi$.  Thus, solving for $w_j(t)$,
   we find that each point of the simple closed curve is a midpoint of
   a balanced hexagon with vertices $w _j(t)$.  Since
   ${\sigma_0(t)}\land{\sigma_2(t)}$ is constant, these balanced
   hexagons have the same area.  Thus, $K$ is a hexameral domain.
\end{proof}

\subsection{star conditions}\label{sec:star}

The convexity of a hexameral domain places certain constraints on the
tangent $X\in\mathfrak{sl}_2(\ring{R})$.  We normalize the curve by
applying an affine transformation so that $\phi(0)=I$ in the circle
representation.  This means that the roots of unity $u^*_j$ lie on the
boundary of the hexameral domain.  We form a hexagram through these
six points.  Specifically, we construct the six equilateral triangles,
each with three vertices:
\[
   u^*_j,\quad u^*_{j+1},\quad (u^*_j + u^*_{j+1}).
\]
For the boundary curve to be convex, the tangent direction $X u^*_j$
at time $t=0$ must lie between the the secant lines joining $u^*_j$
with $u^*_{j\pm 1}$, hence must point into this triangle for each $j$.
If we write
\[
X  = \left(\begin{matrix} a & b \\ c & -a \end{matrix}\right),
\]
we have the following constraints on $X$:
\begin{equation}\label{eqn:star}
\sqrt{3} |a| < c,\quad 3 b + c < 0.
\end{equation}

\begin{lemma}
In this context,
\[\det(X) > 0.\]
\end{lemma}

\begin{proof} 
\[
\det(X) = - b c - a^2 \ge - b c  - \frac{c^2}{3} = \frac{-c(3 b + c)}{3} >0.
\]
\end{proof}

\section{Rank of Multi-Curves}

\begin{definition}[rank]\label{def:rank}
Let $\sigma$ be a multi-curve.  We say that
it has {\it well-defined rank} if the multi-curve is $C^2$,
parametrized by $[t_0,t_1]$, with
the property that for each curve $\sigma_j$ one of the two conditions
hold:
\begin{itemize}
\item It is a line segment.
\item The curvature of $\sigma_j$
is nonzero on the open interval $(t_0,t_1)$.
\end{itemize}
The {\it rank} of such a multi-curve is the number $k\in\{0,1,2,3\}$
of curves $\sigma_{j}$ ($j\in2\rZ/6\rZ$) that are {\it not} line
segments.
\end{definition}
For example, a multi-curve parametrizing a circle
has rank $3$.  The smoothed octagon is parametrized by finitely many
multi-curves of rank $1$.

\begin{lemma}\label{lemma:curvature}  
  Let $\phi:[t_0,t_1]\to SL_2(\ring{R})$ be twice differntiable at $t$.
  Then 
  there exists a $j$ such that the curvature constraint
  (\ref{eqn:curvature''}) is a strict inequality at $t$.
\end{lemma}

\begin{proof}  
    Without of loss of generality, we may take $t=0$ and may apply an affine
  transformation, so that the boundary is given by $\phi u^*_j$, with
  $\phi(0)=I$.  Let $\phi' u^*_j = X\phi u^*_j$.  
The constraint (\ref{eqn:curvature''}) becomes
\[
\phi' u^*_j \land \phi'' u^*_j = X u^*_j \land (X' + X^2) u^*_j.
\]

A short calculation
assuming the vanishing of this curvature inequality for $j=0,2$ gives for $j=4$:
\[
X u^*_4 \land (X' + X^2) u^*_4 = \frac{3\sqrt{3} (a^2 + b c)^2}{3 a + \sqrt{3} c}.
\]
The star conditions in Section~\ref{sec:star} imply that the numerator
and denominator are both positive.
\end{proof}

\begin{lemma}
No multi-curve has rank $0$.
\end{lemma}

\begin{proof} This is a corollary of Lemma~\ref{lemma:curvature}.  A
  direct proof can be given as follows.  The tangent lines to a
  multi-curve, by the argument in the proof of
  Lemma~\ref{lemma:sl2-rein}, determines a balanced hexagon.  If the
  rank is zero, the tangent lines, the balanced hexagon, and its
  midpoints are fixed.  Thus, the curve degenerates to a stationary
  curve at the fixed midpoints.
\end{proof}

It is natural to consider hexameral domains whose boundary is
parametrized by a finite number of  analytic multi-curves.  

\begin{lemma} Suppose that a hexameral domain $K$ has a
  parametrization by a finite number of analytic multi-curves
  $\sigma$.  Then $K$ also admits a parametrization by a finite number
  of triple curves satisfying the hypotheses of
  Definition~\ref{def:rank}, each admitting a well-defined rank.
\end{lemma}

\begin{proof} The curvature of an analytic curve vanishes identically,
or has at most finitely many zeroes on a compact interval $[t_0,t_1]$.
Subdividing the intervals at the finitely many zeroes, we may assume
that the only zeroes appear at the endpoints of the intervals.
\end{proof}

\section{Rank Three and the Ellipse}

The following sections analyze the multi-curves according to rank,
starting with rank three in this section.
The primary method will be the calculus of variations  
to search for a curve $\phi(t)$ in $SL_2(\ring{R})$ that 
minimizes the area of a hexameral domain $K$.  

\subsection{first variation}

We consider a curve $\phi:[t_0,t_1]\to SL_2(\ring{R})$.
Form corresponding curves $\sigma_j(t) = \phi(t) u_j$ and
$\sigma_{j+3}(t) = -\sigma_j(t)$, for $j\in\ring{Z}/6\ring{Z}$ and
$u_j$ satisfying Conditions~(\ref{eqn:uA}).  Consider the closed curve
that follows the line segment from $(0,0)$ to $\sigma_j(t_0)$, the
curve $\sigma_j(t)$ from $t_0\le t\le t_1$, and then the line segment
from $\sigma_j(t_1)$ to $(0,0)$.  Assume that this closed curve is
simple, and let $I_j$ be the area enclosed by the curve.  Set
\[
I(\phi) =\sum_{j=0}^5 I_j.
\]

Let 
\[
\phi(t) = \phi(t_0)\cdot 
\begin{pmatrix} \alpha(t) & \beta(t) \\ \gamma(t) & \delta (t) \end{pmatrix}.
\]
If we express the integrals $I_j$ in terms of $\phi$, a short calculation gives
\begin{equation}\label{eqn:area-int}
I(\phi) = 3\int_{t_0}^{t_1} 
(\alpha d\gamma - \gamma d\alpha) + (\beta d\delta - \delta d\beta).
\end{equation}

\begin{lemma} Let $K$ be a miserly domain.  Pick a $C^1$
  parametrization $\phi(t)\in SL_2(\ring{R})$ of the boundary of $K$.
  Suppose that for all $t\in[t_0,t_1]$, the multi-curve associated
  with the curve $\phi$ has a well-defined rank $3$.  Then the first
  variation of $I(\phi)$ (with fixed boundary conditions) vanishes.
\end{lemma}

\begin{proof} Assume for a contradiction that the first variation is
  nonzero. Under the conditions necessary for the rank to be defined,
  on any compact interval $[s_0,s_1]$ with $t_0 < s_0 < s_1 < t_1$,
  the curvatures of the curves $\sigma_j$ are bounded away from zero.
  Thus, a sufficiently small $C^\infty$-variation of the functional
  preserves the convexity condition.  We can assume the small
  variation gives a simple curve.  By Lemma~\ref{lemma:sl2-rein}, the
  small variation is again the boundary of a hexameral domain.  If the
  first variation is nonzero, the area of can be decreased, holding
  the area of the balanced hexagon constant.  By
  Lemma~\ref{lemma:mid-min}, $K$ is not a miserly domain.
\end{proof}

Basic results about variations now imply that the Euler-Lagrange
equations must hold on $(t_0,t_1)$.  As we are working in
$SL_2(\ring{R})$, we may take the variation of the form
  \begin{equation}\label{eqn:ell-var}
   \exp(\epsilon X(t))\cdot \phi(t)
   \end{equation}
for some curve 
\begin{equation}\label{eqn:X}
 X(t)  = \begin{pmatrix} u(t) & w(t)+v(t) \\ w(t)-v(t) & - u(t)\end{pmatrix}
 \in \mathfrak{sl}_2(\ring{R}),
\end{equation}
the Lie algebra of $SL_2(\ring{R})$.
The Euler-Lagrange equations are
\[
\begin{array}{lll}
0 &= (\delta^2 + \gamma^2)'\\
0 &= (\alpha^2 + \beta^2)'\\
0 &= (\gamma \alpha + \delta\beta)'\\
\end{array}
\]
with boundary conditions
\[
\begin{array}{lll}
1 &= \alpha(t_0) = \delta(t_0),\\
0 &= \beta(t_0) = \gamma(t_0)\\
\end{array}
\]
Integrating, we have 
\[
\begin{array}{lll}
 1 &= \delta^2 + \gamma^2\\
 1 &= \alpha^2 + \beta^2\\
 0 &= \gamma \alpha + \delta \beta.
\end{array}
\]
We also have the determinant condition $\alpha\delta-\beta\gamma=1$.
These are the defining conditions of a special orthogonal matrix.
Hence, there is a function $\theta:[t_0,t_1]\to\ring{R}$ such that
\[
\phi(t) = \phi(t_0)\cdot \begin{pmatrix} \cos\theta(t) & -\sin\theta(t) \\ \sin\theta(t) &
  \cos\theta(t) \end{pmatrix}.
\]
Thus, the curves $\sigma_j$ trace out arcs of an ellipse.
We summarize in the following lemma.

\begin{lemma} Let $K$ be a miserly domain.  Suppose that some portion
  of the boundary has well-defined rank $3$.  Then up to a special
  linear transformation, that portion of the boundary consists of
  three arcs of a unit circle.
\end{lemma}

\subsection{second variation}

Next we study the second variation of the area functional $I(\phi)$.  For this,
we may confine our attention to the unit circle:
\[
\phi(t) = \begin{pmatrix} \cos t & -\sin t \\ \sin t & \cos t
\end{pmatrix}.
\]
Again, we consider variations of the form of
Equations~(\ref{eqn:ell-var}) and (\ref{eqn:X}).  A calculation of the
second variation around $\phi$ gives
\[
\int_{t_0}^{t_1} 4 u(t) w'(t) dt.
\]
This does not have fixed sign.  Therefore, the solution to the
Euler-Lagrange equations is a saddle point.  It follows that an arc of
a circle cannot form part of a miserly domain.  This gives the
following result.

\begin{thm}  Let $K$ be a miserly domain.  Then there is no segment
of the boundary parametrized by a multi-curve of  rank  three.
\end{thm}

\section{Rank Two}

In this section, we consider the variation of a Reinhardt curve of
well-defined rank two.  We prove that the first variation is never
zero in the rank two situation.  This leads to the following theorem.

\begin{thm} Let $K$ be a miserly domain.  Then no segment of the
  boundary is parametrized by a rank two multi-curve.  In fact, the
  first variation in area, along a rank-preserving variation, is
  always non-zero (on the space of multi-curves).
\end{thm}

\begin{proof} By applying a special linear transformation, we may
  assume that $\sigma_0$ moves along the line $y=-1$.  At any given
  time $t$, there is a skew transformation
  \def\xm{\left(\begin{matrix} 1& x\\0&1\end{matrix}\right)}
  \def\vc#1#2{\left(\begin{matrix}#1\\#2\end{matrix}\right)}
\[
\xm
\]
that makes the first coordinates of $\sigma_1(t)$ and $\sigma_2(t)$
equal.  The vertices of the midpoint hexagon (after this
transformation) are
\[
\pm (r-z,-1), \quad \pm (r,s),\quad \pm (r-z,1),
\]
with $r>z>0$, and $-1<s<1$.  This hexagon has area $\sqrt{12}$
provided
\[
z = 2 r - 2 \rho,
\]
where $\rho = \sqrt3/2$.  We regard $r$ as a function of $z$ through
this relation. We may solve for the midpoints of this
hexagon, then reapply the skew transformation to obtain the
coordinates of $\sigma_j$.  We then have
\[
\begin{array}{lll}
\sigma_1(t) &= \frac12\xm\cdot\vc {2\rho}{-1+s},\\
\sigma_2(t) &= \frac12\xm\cdot\vc {2\rho}{+1+s}.\\
\end{array}
\]
The condition that $\sigma$ is a multi-curve forces the tangent to $\sigma_2$
to be an edge of the midpoint hexagon:
\begin{equation}\label{eqn:z}
x' (s^2 - 1) - s' z = 0.
\end{equation}
If $s'(t)=0$, then $x'(t)=0$ and the curve $\sigma_1$ is not regular
at $t$, which is contrary to the definition of a multi-curve.  Thus,
$s'$ is everywhere nonzero, and we can define $z$ in terms of $x$ and
$s$ by the relation (\ref{eqn:z}).  We can pick the sign of $s'$ so
that it is positive.  We have $z >0$.  Then $x'$ is
everywhere negative.

We let $I_j$ be the area bounded by the three curves: the segment from
$0$ to $\sigma_j(t_0)$, the curve $\sigma_j$ on $[t_0,t_1]$, and the
segment from $\sigma_j(t_1)$ to $0$.  A short calculation gives
\[
I_1 + I_2 = \int_{t_0}^{t_1} \frac{1}4 (\sqrt{3} s' - (1+s^2) x')\,dt.
\]
The integral $I_0$ is not relevant for a compactly supported variation
since $\sigma_0$ remains linear under any small rank-preserving
variation, and $I_0$ remains constant.  The first variation of this integral in $x$ clearly does
not vanish, under the established condition $s'>0$.
\end{proof}

\section{Hyperbolic Links (Rank One)}

We have shown that the boundary of a miserly domain cannot
contain multi-curves of rank $\ne 1$.  This section analyses triple
curves of rank one.

\subsection{square representation}

Let $K$ be a hexameral domain with boundary parametrized by a
multi-curve $\sigma$.  Suppose that for some fixed index $j$, the curves
$\sigma_{j+2}$ and $\sigma_{j+4}$ travel along straight lines.  By
applying a special linear transformation, we may assume that
$\sigma_{j+2}$ moves along $x=a$ and $\sigma_{j+4}$ moves along $y=a$,
for some $a>0$.  By reparametrizing the curves, we may assume that
$\sigma_{j+2}(t) = a(1,t)$ and $\sigma_{j+4}(t) = a(s(t),1)$, for some
function $s$.  The fixed area condition on balanced hexagons,
\begin{equation}
{\sigma_{j+2}(t)}\land{\sigma_{j+4}(t)}=\frac{\sqrt{3}}2,
\end{equation}
gives $1- t s = k$, where $k=\sqrt3/(2a^2) > 0$.  This determines the
function $s$.  The condition $\sigma_{j} = -\sigma_{j+2}
-\sigma_{j+4}$ gives $\sigma_{j}(t) = a (-1-s,-1-t)$.  The curve
$\sigma_{j}$ traces the hyperbola $(x+a)(y+a) = a^2(1-k)$, whose
asymptotes are lines $x=-a$ and $y=-a$ containing the curves
$\sigma_{j+5}$ and $\sigma_{j+1}$.  (This calculation shows why
hyperbolic arcs play a special role.)  The curve traced by
$\sigma_{j}$ does not form the arc of a convex region containing the
origin unless
   \begin{equation}
   0 < k < 1,\quad s<0,\quad t < 0.
   \end{equation}
   We assume this condition.  Also, $k<1$ implies $a^2 > \sqrt3/2$.
   The balanced hexagon $G$ degenerates to a quadrilateral when $s \le
   -1$ or $t\le -1$.  We therefore assume that
\begin{equation}
-1<s < 0,\qquad -1 < t < 0.
\end{equation}
  This implies that
\begin{equation}
-1 < s < -t s = k-1,\qquad -1 < t < k-1.
\end{equation}  
The parameter $t$ thus ranges
over an interval $[t_0,t_1]$ with $-1 < t_0 < t_1 < k-1$.  The parameter
$s$ runs over $[s_0,s_1]$ with $s_i = (1-k)/t_i$.  We may
write 
\begin{equation}
t_1 = t_0 + \ta (k-1-t_0),
\end{equation}
for some $\ta\in (0,1)$.

In summary, up to a special linear transformation, and
reparametrization of the curve, a rank one multi-curve is uniquely
determined by the index $j$ of the hyperbolic arc, the initial
parameters $(a,t_0)$, and the terminal parameter $\ta$, where
  \begin{equation}\label{eqn:atu}
  a > \sqrt{3}/2,\quad  k = \sqrt{3}/(2a^2),\quad -1 < t_0 < k - 1,
  \quad 0 < \ta < 1.
  \end{equation}
  We call this the {\it square representation} of the multi-curve.
  The starting and terminal point $s_i$ for the curve
  $\sigma_{j+4}(t_i) = a(s_i,1)$ and the constant $k$ are determined
  by Equation~(\ref{eqn:AB}).  The curve $\sigma_{j}$, is determined
  by $\sigma_{j+2}$ and $\sigma_{j+4}$.

Conversely, if we are given three parameters $a,t_0,\ta$ satisfying
the conditions (\ref{eqn:atu}), and given the hyperbolic index $j$, 
there exists a multi-curve whose
square representation has these parameters.  This gives us a convenient
way to construct multi-curves.  

\subsection{area}

Let $I_j(a,t_0,t_1)$ be the area bounded by the line segment from
$(0,0)$ to $\sigma_j(t_0)$, the curve $\sigma_j$, $t_0\le t\le t_1$,
and the line segment from $\sigma_j(t_1)$ to $(0,0)$.  An easy
calculation gives
\begin{equation}\label{eqn:I}
\begin{array}{lll}
 I(a,t_0,t_1) &=
  a^2((s_0-s_1)+(t_1-t_0) - (1-k)\ln (s_1/s_0))\\
 &= a^2 ((1-k) (1/t_0-1/t_1) + (t_1-t_0) -(1-k) \ln (t_0/t_1)),
\end{array}
\end{equation}
where 
$I=I_0+I_2+I_4,$ and the parameters $s_1,s_2,k$ are given in terms
of $a,t_0,t_1$ as above.

For example, the boundary of the 
smoothed octagon is parametrized by eight multi-curves (one for
each hyperbolically rounded corner of the octagon).  The parameters
for each of the eight multi-curves of the smoothed octagon are
\[
\begin{array}{lll}
a &= \frac{12^{1/4}}{\sqrt{4-\sqrt{2}}},\\
t_0 &= -1/\sqrt{2},\\
t_1 &= -1/2,\\
I &= \frac{\sqrt{3} \left(8-8 \sqrt{2}+\sqrt{2} \log (2)\right)}{4
   \left(-4+\sqrt{2}\right)}\\
\end{array}
\]
This gives density
\[8 I/\sqrt{12} \approx 0.902414,\]
mentioned in the introduction to this article.

\subsection{the set of initial states}

The initial state for a multi-curve is specified by a matrix
$\phi(t_0)\in SL_2(\ring{R})$ and a velocity
$X(t_0)\in\mathfrak{sl}_2(\ring{R})$, given by
Equation~(\ref{eqn:Xt}).  We wish to allow reparametrizations of the
curve, so that the velocity is given only up to a scalar, giving a
point in projective space:
$[X(t_0)]\in\ring{P}(\mathfrak{sl}_2(\ring{R}))$.  The space of
initial states then has dimension five:
\[
\dim\, S = 3 + 2,\quad\hbox{ where } 
S = SL_2(\ring{R}) \times\ring{P}(\mathfrak{sl}_2(\ring{R})).
\]
These five dimensions correspond to the three dimensional group of
transformations that can be used to transform a rank-one multi-curve
to its square representation, together with the two parameters $(a,t_0)$
giving the initial state in the square representation.

Likewise, the terminal state for a multi-curve is given by a point
in the same five dimensional space.

\subsection{hyperbolic chains and smoothed polygons}

\begin{definition}
  An analytic multi-curve of rank one is called a {\it hyperbolic
    link} (because the of hyperbolic arc $\sigma_j$).  A piecewise
  analytic multi-curve of rank one is called a {\it hyperbolic chain}.
  A hexameral domain $K$ whose boundary is a hyperbolic chain is is
  called a smoothed polygon.
\end{definition}

We are now in a position to state the main result of this article.

\begin{thm}  Let $K$ be a miserly domain.  If the boundary of $K$ is piecewise analytic,
then $K$ is a smoothed polygon.
\end{thm}

\begin{proof}  
  In earlier sections, we have ruled out the existence of segments on
  the boundary of ranks zero, two, or three.  Thus, it must consist of
  segments of rank one.
\end{proof}

We consider a multi-curve that consists of a finite number of rank one
triples, joined one to another to form $C^1$ curves.  There is no
variational problem here, because there are no functional degrees of
freedom for segments of rank one.  When the an initial state for a
hyperbolic link is fixed, there is exactly one degree of freedom, the
parameter $\ta\in(0,1)$ in the square representation.

Suppose that we have a multi-curve $\sigma$ consisting of a finite
number of hyperbolic links.  Along each hyperbolic link, exactly one
of the three curves $\sigma_j$ follows a hyperbolic arc; the other two
are linear.  Set $\Delta=2\ring{Z}/6\ring{Z}$.  We break the domain of
the multi-curve into finitely many subintervals, each labeled with an
index $j\in\Delta$, according to which arc $\sigma_j$ is the
hyperbola.  The first link of $\sigma$ is entirely specified by
$(s_1,\ta_1,j_1)$, where $s_1\in S$ is an initial state,
$\ta_1\in(0,1)$ determines the length of the arc, and $j_1\in\Delta$
specifies the index of the hyperbolic arc.  The triple 
$(s_1,\ta_1,j_1)$ determines the terminal state $s_2\in S$ of the first
hyperbolic link.  Similarly, the second link of $\sigma$ is determined
by $(s_2,\ta_2,j_2)$, for some $\ta_2\in (0,1)$ and $j_2\in\Delta$.
Working through the hyperbolic chain, link by link, we obtain a
sequence
\begin{equation}\label{eqn:param}
  s_0\in S,\quad ((\ta_0,j_0),(\ta_1,j_1),\ldots,
(\ta_n,j_n)),\quad \ta_i\in(0,1),\quad j_i\in\Delta
\end{equation}
that uniquely determines the hyperbolic chain.

Conversely, given a sequence of parameters (\ref{eqn:param}), an
induction on $n$ shows that there is a unique hyperbolic chain with
those parameters.  We can extend the parameters $\ta\in (0,1)$ to
include $\ta=0$, with the understanding that this corresponds to a
degenerate link consisting of a single point.  If two consecutive
parameters $(\ta_i,j)$ and $(\ta_{i+1},j)$ have the same index
$j\in\Delta$, then they can be combined into a single hyperbolic link.
Thus, there is no loss in generality in assuming that consecutive
links carry distinct hyperbolic indices $j$.  In fact, we may insert
degenerate parameters $(\ta,j)$, with $\ta=0$ so that the parameters
take the special form
\begin{equation}\label{eqn:ji}
  (\ta_i,j_i),\hbox{ where } j_i = j_0 + 2 i.
\end{equation}

If we are given a hyperbolic chain $\sigma$ 
with parameters (\ref{eqn:param}),
we may extract the square representation $(a(i),t_0(i),t_1(i))$
of each
link $i$ from these parameters.  We may then sum Equation (\ref{eqn:I})
over the set of links to obtain the area
\begin{equation}\label{eqn:Isigma}
I(\sigma) = I(s_0,((u_0,j_0),\ldots)) = \sum_i I(a(i),t_0(i),t_1(i))
\end{equation}
represented by the entire chain.

\subsection{closed curves}

Consider a multi-curve $\sigma$ that gives the boundary of a hexameral
domain.  In the circle representation, write $\sigma_j(t) =
\phi(t)u^*_j$.  We have
\[
u^*_{j+1} = \rho u^*_j,\quad \rho=
\left(\begin{array}{ccc} \cos\theta & -\sin\theta\\ 
\sin\theta & \cos\theta\end{array}\right),\quad \theta=\pi/3.
\]
Let the initial point of the curve be given at time $t_0$ and let
$t_1>t_0$ be the first time at which $\sigma_0(t_1) = \sigma_1(t_0)$:
\begin{equation}\label{eqn:t0}
t_1 = \min \{ t \mid \sigma_0(t) = \sigma_1(t_0), \quad t>t_0\}.
\end{equation} 
Then for the boundary of the hexameral domain to be closed
and $C^1$, we must have
\begin{equation}\label{eqn:closed}
\phi(t_1) = \phi(t_0)\rho\quad\hbox{ and }\quad X(t_1) = X(t_0).
\end{equation}
The angular argument must also satisfy
\begin{equation}\label{eqn:arg}
0\le \arg (\phi(t_0)^{-1}\phi(t)u^*_0) \le \pi/3,\quad t\in[t_0,t_1].
\end{equation}

Given a choice $u$ of multi-point on the boundary of $K$, the parameters
$s\in S$ and $(\ta_i,j_i)$ are uniquely determined.  Conversely, the
parameters uniquely determine the boundary of $K$.

When we choose to do so, we may pick the starting multi-point on
the boundary of $K$ in such a way
that the first parameter is
\begin{equation}\label{eqn:j0}
(\ta_0,j_0) \hbox{ with } j_0=0.
\end{equation}  
We may also arrange
that the multipoint is the endpoint of a hyperbolic link.

\begin{definition}[link length]
Assume that the boundary of a hexameral domain is a hyperbolic
chain that satisfies conditions (\ref{eqn:j0}), (\ref{eqn:ji}).
Let $t_0,t_1$ be given as in (\ref{eqn:t0}).  We may extend
$\sigma_j$ to a periodic function $\rR\to\rR^2$.  We may extend
the parameters $(\ta_i,j_i)$ to all of $i\in\rZ$ with $j_i = 2i$.
The curve $\sigma_j$, restricted to $[t_0,t_1]$ has parameters
\[
((\ta_0,0),(\ta_1,2),\ldots,(\ta_n,2n)),
\]
for some $n$.  When $n$ is chosen to give the shortest representation
of this form, we call $n+1$ the {\it link length} of the hexameral
domain.
\end{definition}

For example, the boundary of the smoothed octagon contains eight
hyperbolic arcs, one at each corner of the octagon. They appear in
centrally symmetric pairs.  There exists a choice of initial
multi-point such that the smoothed octagon is by the following data:
\[
(I,[X])\in S,\quad ((\ta,0),(\ta,2),(\ta,4),(\ta,0)),
\]
for some $\ta>0$ that is independent of the link and some
$X\in\mathfrak{sl}_2(\rR)$.
The link length is $4$.

\begin{lemma}\label{lemma:link}  
  Let $K$ be a hexameral domain whose boundary is a hyperbolic chain.
  Let $n+1$ be the link length of $K$.  Then
\[n\equiv 0 \mod 3.\]
(In particular, the smoothed octagon minimizes the link length.)
\end{lemma}

\begin{proof} If we consider the multi-curve $\sigma$, the link
$(\ta_{n+1},2n+2)$ lies along the same part of the multi-curve as $(\ta_0,0)$,
so that $\ta_{n+1}=\ta_0$.  However, the hyperbolic index shifts by
$1$ as we pass from $\sigma_0$ to $\sigma_1$.  Therefore,
\[(\ta_{n+1},2n+2) = (\ta_0,2)\in \rR\times \Delta.\]
The congruence 
\[2n+2 \equiv 2\in \Delta = 2\rZ/6\rZ\]
gives the result.
\end{proof}

\subsection{link reduction}

With this background in place, we now return to a discussion of the
Reinhardt conjecture.  Lemma~\ref{lemma:link} shows that the
conjectural solution to the Reinhardt problem minimizes link
length.  This leads to the intuition 
that decreases in $\deltalat(K)$ should have
the effect of shortening the link length.  This is the motivation
for the Conjecture~\ref{conj:defrag}, which asserts that we may
simultaneously decrease areas and eliminate links
from some hyperbolic chains.   

We consider all possible hyperbolic chains $\sigma$ with the same
initial and terminal states $s_0,s_1$.  These chains are given by
parameters
  \[
  s_0,\quad ((\ta_1,j_1),\ldots).
  \]
  Since $s_1$ lies in a five-dimensional space, the terminal state
  places five constraints on the parameters set
  $(\ta_1,\ta_2,\ldots)$.  By counting dimensions, we might guess that
  for generic parameters $s_0,s_1$, the terminal state cuts out a set
  of hyperbolic chains of codimension five.  Generically, it should
  take at least five links to match the terminal state $s_1$.

If, instead of fixing both endpoints, we may impose the closed
curve condition (\ref{eqn:closed}). We may use the action of
the special linear group to force $\phi(t_0)=I$.  The free
parameters are $X$ and $(\ta_0,\ldots,\ta_n)$, or $n+3$ free
variables.

\begin{conj}[Link Reduction] \label{conj:defrag} Let $K$ be a
  hexameral domain with multi-curve $\sigma$ around the boundary.  Let
  $t_0,t_1$ be the parameters (\ref{eqn:t0}).  Suppose that a portion
  $[t'_0,t'_1]$ of the boundary (with $t_0\le t'_0\le t'_1\le t_1$) is
  a hyperbolic chain with six links, given in the form
  (\ref{eqn:param}).  (We do not assume (\ref{eqn:j0}),
  (\ref{eqn:ji}).)  Let $s_0,s_1\in S$ be the initial and terminal
  states for the curve at $t_0'$ and $t_1'$.  Then there is
  another hyperbolic chain with five links that
\begin{itemize}
\item fits the same initial and terminal states $s_0,s_1\in S$,
\item satisfies the angle condition (\ref{eqn:arg}),
\item has no greater area $I(\sigma)$ as defined by (\ref{eqn:Isigma}),
\item and in fact has strictly smaller area, unless the chain is already (degenerately) a 
five-link chain.
\end{itemize}
\end{conj}

In other words, the conjecture claims that one of the links can
be removed, decreasing the number of links to five, while
simultaneously decreasing area.

The conditions on the parameters $t_0,t_1,t'_0,t'_1$ are there to
insure that the hyperbolic chain is short enough that the
corresponding curves $\sigma_0,\ldots,\sigma_5$ parametrize distinct
portions of the boundary.

The condition that the hyperbolic chain should bound part of a hexameral
domain places constraints on $s_0,s_1$.  They cannot be arbitrary
states of $S$.

The number of parameters in the implied optimization is eight: six
links and the five-dimensional initial state, reduced by the three
dimensional group $SL_2(\rR)$.  This conjecture may be explored by
computer, but I have only done so in a very limited way.

\bigskip

The following is a very special case of the Reinhardt conjecture.

\begin{conj}[Five Link]\label{conj:5} 
  Let $K$ be a hexameral domain with multi-curve $\sigma$ around the
  boundary.  Let $t_0,t_1$ be the parameters (\ref{eqn:t0}).  Suppose
  that the entire non-repeating boundary $[t_0,t_1]$ is a hyperbolic
  chain with five links.  Then $\deltalat(K)= \delt$ exactly when $K$
  is the smoothed octagon, up to a transformation by $SL_2(\rR)$.
\end{conj}

The smoothed octagon belongs to this family of hexameral
domains, with parameters $\ta_3=0$ and $\ta_i=\ta_j$, if $i,j\ne 3$.

Nazarov's proof of the local optimality of the smoothed octagon
gives the conjecture for hexameral domains sufficiently close
to the smoothed octagon \cite{Nazarov}.

This is an optimization problem on a seven dimensional space.  (There
is the five dimensional initial state $s\in S$ and five links
$(\ta_i,j_i)$, reduced by the action of the three-dimensional group
$SL_2$.)  For a generic choice of parameters $s,\{(\ta_i,j_i)\}$ the
hyperbolic chain will not form a simple closed curve, and can be
discarded.

Again, we might hope to test this conjecture by computer, and perhaps
even to prove it with interval arithmetic.

\begin{lemma}\label{lemma:smooth}  
  Assume Conjectures~\ref{conj:defrag} and~\ref{conj:5}.  Then up to
  affine transformation, the smoothed octagon uniquely minimizes
  $\deltalat(K)$ over the class of all smoothed polygons $K$.
\end{lemma}

\begin{proof} The first conjecture successively reduces the number of links to
  five.  The second conjecture treats the case of five links (which
  includes as degenerate cases, fewer than five links).
\end{proof}

\section{Bolza}

Viewed as a problem in the calculus of variations or control theory,
the Reinhardt problem is an instance of the classical problem of Bolza
with nonholonomic inequality constraints, autonomous, fixed endpoints,
and no isoperimetric constraints.  Lacking isoperimetric constraints,
it is an instance of the classical problem of Mayer with inequality
constraints~\cite[Ch.7]{Hestenes:1966}.

We have established the existence of a minimizer with Lipschitz
continuous derivative (when considered as a second-order system; the
minimizer itself is Lipschitz continuous when converted to a
first-order system).  This is sufficiently regular to match the
hypotheses in standard treatments of the subject, such as~\cite{Pontryagin:1986}.

We search for minimizers of the integral (\ref{eqn:area-int}):
\begin{equation}
I(\phi) = 3\int_{t_0}^{t_1} (\alpha d\gamma - 
\gamma d\alpha) + (\beta d\delta - \delta d\beta).
\end{equation}
over the class of curves $\phi:[t_0,t_1]\to SL_2(\ring{R})$ such that 
\begin{itemize}
\item the special linear condition holds: $\alpha\delta-\beta\gamma=1$, where
 \[
\phi(t) = \phi(t_0)
\left(\begin{matrix}\alpha&\beta\\\gamma&\delta\end{matrix}\right).
\]
\item $\phi$ is $C^1$
\item The derivative $\phi'$ is Lipschitz continuous.
\item Convexity constraints hold:
  \[\sigma_j'(t) \land \sigma_j''(t) \ge 0,\quad j=0,2,4,\quad t\in[t_0,t_1]~a.e.\]
where $u^*_j$ is as in (\ref{eqn:roots}) and $\phi(t_0)\sigma_j(t) = \phi(t) u^*_j$.
\end{itemize}

The variational problem has several symmetries, which lead to
conserved quantities by Noether's theorem.  The convexity
constraints and area are invariant under reparametrization of
$\phi:[t_0,t_1]\to SL_2(\ring{R})$.  The problem is autonomous.
Finally, the group $SL_2(\ring{R})$ acts on the set of minimizers.

We may convert to a first order problem by setting $\phi(t_0)Y =
\phi'$.  The convexity constraints become linear in $Y'$:
\begin{equation}\label{eqn:curvature}
  y_j \land y_j' \ge 0,
\end{equation}
where $y_j = Y u^*_j$.

\raggedright

\bibliographystyle{plain} 


\end{document}